\documentclass{amsart}

\usepackage{amsmath, amsthm, amssymb, amstext, amsfonts}
\usepackage{enumerate}
\usepackage{graphicx}
\usepackage[applemac]{inputenc}

\theoremstyle{plain}

\newtheorem{thm}{Theorem}[section]
\newtheorem{lem}[thm]{Lemma}

\theoremstyle{definition}

\newcommand{\sm}{\ensuremath{\smallsetminus}}

\newcommand{\isom}{\ensuremath{\cong}}

\newcommand{\es}{\ensuremath{\emptyset}}

\newcommand{\sub}{\subseteq}

\newcommand{\comment}[1]{}

\newcommand{\nat}{{\mathbb N}}

\newenvironment{txteq*}
  {
    \begin{equation*}
    \begin{minipage}[c]{0.85\textwidth} 
    \em                                
  }
  {\end{minipage}\end{equation*}\ignorespacesafterend}

\begin{document}

\title[Homogeneous $2$-partite digraphs]{\boldmath Homogeneous $2$-partite digraphs}
\author{Matthias Hamann}
\address{Matthias Hamann, Department of Mathematics, University of Hamburg, Bundes\-stra\ss e~55, 20146 Hamburg, Germany}
\date{}
\maketitle

\begin{abstract}
We call a $2$-partite digraph $D$ \emph{homogeneous} if every isomorphism between finite induced subdigraphs that respects the $2$-partition of~$D$ extends to an automorphism of~$D$ that does the same.
In this note, we classify the homogeneous $2$-partite digraphs.
\end{abstract}

\section{Introduction}

A structure is \emph{homogeneous} if every isomorphism between finite induced substructures extends to an automorphism of the whole structure.
This notion is due to Fra\"iss\'e~\cite{Fraisse1953}, see also~\cite{F-TheoryOfRelations}.
Since his work appeared, several countable homogeneous structures have been classified.
These classification results include partial orders by Schmerl~\cite{Schmerl-HomogeneousPO}, graphs by Gardiner~\cite{Gard-HomogeneousGraphs} and by Lachlan and Woodrow~\cite{LW-CountUltrahomGraphs}, tournaments by Lachlan~\cite{L-Tournaments}, directed graphs by Lachlan~\cite{L-FiniteHomDigraphs} and Cherlin~\cite{Cherlin-HomImprimitive,Cherlin-CountHomDigraphs}, bipartite graphs by Goldstern, Grossberg, and Kojman~\cite{GGK}, and, recently, ordered graphs by Cherlin~\cite{Cherlin-OrderedGraphs}.
For more details on homogeneous structures, we refer to Macpherson's survey~\cite{Macpherson-Survey}.

In this note, we classify the homogeneous $2$-partite digraphs (Theorem~\ref{thm_main}).
This classification problem occured during the classification of the countable connected-homogeneous digraphs~\cite{CountConHomDi}, where a digraph is \emph{connected-homogeneous} if every isomorphism between finite induced connected subdigraphs extends to an automorphism of the whole digraph.

\section{Preliminaries}\label{sec_basics}

In this note, a \emph{bipartite graph} is a triple $G=(X,Y,E)$ of pairwise disjoint sets such that every $e\in E$ is a set consisting of one element of~$X$ and the one element of~$Y$.
We call $VG=X\cup Y$ the \emph{vertices} of~$G$ and $E$ the \emph{edges} of~$G$.
A \emph{$2$-partite digraph} is a triple $D=(X,Y,E)$ of pairwise disjoint sets with $E\sub (X\times Y)\cup (Y\times X)$ and such that $(u,v)\in E$ implies $(v,u)\notin E$.
Again, $VD=X\cup Y$ are the \emph{vertices} of~$D$ and $E$ are the \emph{edges} of~$D$.
We write $uv$ instead of~$(u,v)$ for edges of~$D$.
A~$2$-partite digraph $(X,Y,E)$ is \emph{bipartite} if either $E\sub X\times Y$ or $E\sub Y\times X$.
The \emph{underlying undirected bipartite graph} of a $2$-partite digraph $(X,Y,E)$ is defined by
\[(X,Y,\{\{u,v\}\mid uv\in E\}).\]

Two vertices $u,v$ of a $2$-partite digraph $D=(X,Y,E)$ are \emph{adjacent} if either $uv\in E$ or $vu\in E$.
The \emph{successors} of~$u\in VD$ are the elements of the \emph{out-neighbourhood} $N^+(u):={\{w\in VD\mid uw\in E\}}$ and its \emph{predecessors} are the elements of the \emph{in-neighbourhood} $N^-(u):=\{w\in VD\mid wu\in E\}$.
For $x\in X$, we define
\[x^\perp=\{y\in Y\mid y\text{ not adjacent to }x\}\]
and, for $y\in Y$, we define
\[y^\perp=\{x\in X\mid x\text{ not adjacent to }y\}.\]

A bipartite graph $G=(X,Y,E)$ is \emph{homogeneous} if every isomorphism $\varphi$ between finite induced subgraphs $A$ and~$B$ with $(VA\cap X)\varphi\sub X$ and $(VA\cap Y)\varphi\sub Y$ extends to an automorphism $\alpha$ of~$G$ with $X\alpha=X$ and $Y\alpha=Y$.
Similarly, a $2$-partite digraph $D=(X,Y,E)$ is \emph{homogeneous} if every isomorphism $\varphi$ between finite induced subdigraphs $A$ and~$B$ with $(VA\cap X)\varphi\sub X$ and $(VA\cap Y)\varphi\sub Y$ extends to an automorphism $\alpha$ of~$D$ with $X\alpha=X$ and $Y\alpha=Y$.

A first step towards the classification of the homogeneous $2$-partite digraphs was already done when Goldstern et al.~\cite{GGK} classified the homogeneous bipartite graphs.
Thus, before moving on, we cite their result and discuss its effects towards the classification of the homogeneous $2$-partite digraphs.

\begin{thm}\label{thm_GGK}{\rm \cite[Remark 1.3]{GGK}}
A bipartite graph is homogeneous if and only if it is isomorphic to one of the following bipartite graphs:
\begin{enumerate}[{\rm (i)}]
\item a complete bipartite graph;
\item an empty bipartite graph;
\item a perfect matching;
\item the bipartite complement of a perfect matching;
\item a generic bipartite graph.\qed
\end{enumerate}
\end{thm}

The \emph{bipartite complement of a perfect matching} is a complete bipartite graph with sides of equal cardinality where a perfect matching is removed from the edge set.
A bipartite graph $G=(X,Y,E)$ is \emph{generic} if for each two disjoint finite subsets $U_X,W_X$ of~$X$ and each two disjoint finite subsets $U_Y,V_Y$ of~$Y$ there exist $y\in Y$ and $x\in X$ with $U_X\sub N(y)$ and $V_X\cap N(y)=\es$ as well as with $U_Y\sub N(x)$ and $V_Y\cap N(x)=\es$.

\medskip

For bipartite digraphs $(X,Y,E)$, Theorem~\ref{thm_GGK} applies analogously in the following sense:
as we have either $E\sub X\times Y$ or $E\sub Y\times X$, the underlying undirected bipartite graph is homogeneous, so belongs to some class of the list in Theorem~\ref{thm_GGK}.
Conversely, every orientation of a homogeneous bipartite graph that results in a bipartite digraph gives a homogeneous bipartite digraph.
Note that homogeneous bipartite digraphs are in particular homogeneous $2$-partite digraphs.
Hence, the above classification gives us a partial classification in the case of the homogeneous $2$-partite digraphs in that it gives a full classification of the homogeneous bipartite digraphs.
In the remainder of this note we extend this partial classification by classifying those homogeneous $2$-partite digraphs that are not bipartite.

\section{The main result}\label{sec_Hom2PartiteDigraphs}

In this section, we shall prove our main theorem, the classification of the homogeneous $2$-partite digraphs (Theorem~\ref{thm_main}).

\begin{thm}\label{thm_ClassHom2Partite}\label{thm_main}
A $2$-partite digraph is homogeneous if and only if it is isomorphic to one of the following $2$-partite digraphs:
\begin{enumerate}[\rm (i)]
\item a homogeneous bipartite digraph;
\item an $M_\kappa$ for some cardinal $\kappa\geq 2$;
\item a generic $2$-partite digraph;
\item a generic orientation of a generic bipartite graph.
\end{enumerate}
\end{thm}

For a cardinal $\kappa\geq 2$, let $M_\kappa$ be a bipartite digraph $(X,Y,E)$ with $|X|=\kappa=|Y|$ such that either $(X,Y,E\cap (X\times Y))$ or $(X,Y,E\cap (Y\times X))$ is a perfect matching and the other is the bipartite complement of a perfect matching.
In particular, the underlying undirected bipartite graph is a complete bipartite graph.

We call a $2$-partite digraph $(X,Y,E)$ \emph{generic} if its underlying undirected bipartite graph is a complete bipartite graph and if for all pairwise disjoint finite subsets $A_X,B_X\sub X$ and $A_Y,B_Y\sub Y$ there are vertices $y\in Y$ and $x\in X$ with $A_X\sub N^+(y)$ and $B_X\sub N^-(y)$ as well as $A_Y\sub N^+(x)$ and $B_Y\sub N^-(x)$.
Similarly, we call a $2$-partite digraph $(X,Y,E)$ a \emph{generic orientation of a generic bipartite graph} if for all pairwise disjoint finite subsets $A_X,B_X,C_X\sub X$ and $A_Y,B_Y,C_Y\sub Y$ there are vertices $y\in Y$ and $x\in X$ with $A_X\sub N^+(y)$, $B_X\sub N^-(y)$ and $C_X\sub y^\perp$ as well as with $A_Y\sub N^+(x)$, $B_Y\sub N^-(x)$ and $C_Y\sub x^\perp$.
It is easy to verify that its underlying undirected graph is a generic bipartite graph.

Note that standard back-and-forth arguments show that, up to isomorphism, there are a unique countable generic $2$-partite digraph and a unique countable generic orientation of the (unique) countable generic bipartite graph.

\medskip

It is worthwhile noting that by Theorem~\ref{thm_main} the underlying undirected bipartite graph of a homogeneous $2$-partite digraph is always homogeneous, which is false for arbitrary homogeneous digraphs and their underlying undirected graphs.

The fact that the listed  $2$-partite digraphs in Theorem~\ref{thm_main} are homogeneous is already discussed in the previous section for case (i), while in case (ii) it is a consequence of the fact that the bipartite complement of a perfect matching is homogeneous.
The cases (iii) and (iv) can be easily verified by the above mentioned back-and-forth argument.
(This can also be applied if they are not countable to show that they are homogeneous.)
Before we start with the remaining direction of the proof of Theorem~\ref{thm_main}, we show some lemmas.

\begin{lem}\label{lem_hom2partite2}
Let $D=(X,Y,E)$ be a homogeneous $2$-partite digraph.
If $N^+(v)$ and $N^-(v)$ are infinite and $v^\perp$ is finite for some $v\in VD$, then $v^\perp=\es$.
\end{lem}

\begin{proof}
Let $x\in X$.
First, let us suppose that $m:=|x^\perp|=1$.
We note that any automorphism of~$D$ that fixes~$x$ must also fix the unique element $x_Y\in x^\perp$.
Indeed, since $D$ is homogeneous and each of the two sets $\{y_1,y_2\}$ and $\{y,x_Y\}$ induces a digraph without any edge, we can extend every isomorphism between them to an automorphism $\alpha$ of~$D$ and, if x' is the common predecessor of~$y_1$ and~$y_2$, then $x'\alpha$ is the common predecessor of~$y$ and~$x_Y$.
Let $y$ be a successor of~$x$.
As $N^+(x)$ is infinite, we find two vertices $y_1,y_2$ in~$Y$ that have a common predecessor.
Homogeneity then implies that the two vertices $y$ and~$x_Y$ in~$Y$ have a common predecessor~$z$.
Let $z'$ be a successor of~$x_Y$.
By homogeneity, we find an automorphism $\beta$ of~$D$ that fixes~$x$ and maps $z$ to~$z'$.
As mentioned above, $\beta$ must fix~$x_Y$ as it fixes~$x$.
But we have $zx_Y\in E$ and $(x_Yz)\alpha=x_Yz'\in E$, which is impossible.

Now let us suppose that $|x^\perp|\geq 2$.
By homogeneity and as~$m$ is finite, we find for any subset $A$ of~$Y$ of cardinality~$m$ a vertex $a\in X$ with $a^\perp=A$.
As $Y$ is infinite, there are two subsets $A_1,A_2$ of~$Y$ of cardinality~$m$ with $|A_1\cap A_2|=m-1$ and two such subsets $B_1,B_2$ with $|B_1\cap B_2|=m-2$.
Let $a_i, b_i\in X$ with $a_i^\perp=A_i$ and $b_i^\perp=B_i$, respectively.
Then there is no automorphism of~$D$ that maps $a_1$ to~$b_1$ and $a_2$ to~$b_2$ even though $D$ is homogeneous as the number of vertices that are not adjacent to~$a_1$ and~$a_2$ is larger than the corresponding number for~$b_1$ and~$b_2$.
Analogous contradictions for any vertex in~$Y$ instead of $x\in X$ show the assertion.
\end{proof}

\begin{lem}\label{lem_hom2partite1}
Let $D=(X,Y,E)$ be a homogeneous $2$-partite digraph.
If $N^+(v)$ and $N^-(v)$ are infinite and $v^\perp=\es$ for all $v\in VD$, then $D$ is a generic $2$-partite digraph.\looseness-1
\end{lem}

\begin{proof}
It suffices to show that for any two disjoint finite subsets $A$ and~$B$ of~$X$ we find a vertex $v\in Y$ with $A\sub N^+(v)$ and $B\sub N^-(v)$.
Indeed, the corresponding property for subsets of~$Y$ then follows analogously.
Note that we find for every $y\in Y$ two sets $A_y\sub N^+(y)$ and $B_y\sub N^-(y)$ with $|A|=|A_y|$ and $|B|=|B_y|$.
As~$D$ is homogeneous and as $A\cup B$ and $A_y\cup B_y$ induce (empty) isomorphic finite subdigraphs of~$D$, there exists an automorphism $\alpha$ of~$D$ that maps $A_y$ to~$A$ and $B_y$ to~$B$.
So $y\alpha$ is a vertex we are searching for.
\end{proof}

\begin{lem}\label{lem_hom2partite3}
Let $D=(X,Y,E)$ be a homogeneous $2$-partite digraph.
If $N^+(v)$, $N^-(v)$, and $v^\perp$ are infinite for all $v\in VD$, then $D$ is a generic orientation of a generic bipartite graph.
\end{lem}

\begin{proof}
Similarly to the proof of Lemma~\ref{lem_hom2partite1}, it suffices to show that for any three pairwise disjoint finite subsets $A,B,C$ of~$X$ we find a vertex $v\in Y$ with $A\sub N^-(v)$ and $B\sub N^+(v)$ and $C\sub v^\perp$.
For every $y\in Y$, we find subsets $A_y\sub N^+(y)$ and $B_y\sub N^-(y)$ and $C_y\sub y^\perp$ with $|A|=|A_y|$ and $|B|=|B_y|$ and $|C|=|C_y|$.
Note that each of the two sets $A\cup B\cup C$ and $A_y\cup B_y\cup C_y$ has no edge.
Applying homogeneity, we find an automorphism $\alpha$ of~$D$ that maps $A_y$ to~$A$ and $B_y$ to~$B$ and $C_y$ to~$C$.
So $y\alpha$ is a vertex that has the desired properties.
\end{proof}

Now we are able to prove our main theorem.

\begin{proof}[Proof of Theorem~\ref{thm_ClassHom2Partite}]
Let $D=(X,Y,E)$ be a homogeneous $2$-partite digraph that is not bipartite.
Then we find in~$X$ some vertex with a predecessor in~$Y$ and some vertex with a successor in~$Y$.
By homogeneity, we can map the first onto the second and conclude the existence of a vertex in~$X$ that has a predecessor and a successor in~$Y$.
Analogously, we obtain the same for some vertex of~$Y$.
By homogeneity, every vertex of~$D$ has predecessors and successors.
In particular, we have $|X|\geq 2$ and $|Y|\geq 2$.\looseness-1

Let us suppose that two vertices $u,v\in X$ have the same successors, that is, $N^+(u)=N^+(v)$.
By homogeneity, we can fix~$u$ and map $v$ onto any vertex~$w$ of~$X\sm\{u\}$ by some automorphism of~$D$ and thus obtain $N^+(w)=N^+(u)$ for every $w\in X$.
So no vertex in $N^+(u)$ has successors in~$X$, which is impossible as we saw earlier.
Hence, we have $N^+(u)\neq N^+(v)$ for each two distinct vertices $u,v\in X$.
Analogously, the same holds for each two distinct vertices in~$Y$ and also for the set of predecessors of every two vertices either in~$X$ or in~$Y$.
Thus, we have shown
\begin{equation}\label{eq_1}
N^+(u)\neq N^+(v)\quad \text{and}\quad N^-(u)\neq N^-(v)\quad\text{for all }u\neq v\in X
\end{equation}
and
\begin{equation}\label{eq_2}
N^+(u)\neq N^+(v)\quad \text{and}\quad N^-(u)\neq N^-(v)\quad\text{for all }u\neq v\in Y.
\end{equation}

{\lineskiplimit=-3pt
Let us assume that $n:=|N^+(u)|$ is finite for some $u\in X$.
Note that, for any subset $A$ of~$Y$ of cardinality~$n$, we find a vertex $a\in X$ with $N^+(a)=A$ by homogeneity.
If $|Y|>n+1$ and $n\geq 2$, then we find two subsets of~$Y$ of cardinality~$n$ whose intersection has $n-1$ elements and two such sets whose intersection has $n-2$ elements.
So we find two vertices in~$X$ with $n-1$ common successors and we also find two vertices in~$X$ with $n-2$ common successors.
This is a contradiction to homogeneity, because we cannot map the first pair of vertices onto the second pair.
Thus, we have either $n=1$ or $|Y|=n+1$.
If $|Y|=n+1$, then we directly obtain ${D\isom M_{n+1}}$ since every vertex in~$X$ also has some predecessor in~$Y$.
So let us assume ${n=1}$.
If we have $1<k\in\nat$ for $k:=|N^-(u)|$, then we obtain $D\isom M_{k+1}$, analogously.
So let us assume that either $|N^-(u)|=1$ or $N^-(u)$ is infinite.
First, we consider the case that $N^-(u)$ is infinite.
An empty set $u^\perp$ directly implies ${D\isom M_{|Y|}}$.
So let us suppose $u^\perp\neq\es$.
Let $u^+$ be the unique vertex in $N^+(u)$.
Since $u^\perp\neq\es$, we find for some and hence by homogeneity for every vertex in~$Y$ some vertex in~$X$ it is not adjacent to.
Let $w\in (u^+)^\perp$ and let $v\in N^+(u^+)$.
By homogeneity, we find an automorphism $\alpha$ of~$D$ that fixes $u$ and maps $v$ to~$w$.
Since $\alpha$ fixes $u$, it must also fix $u^+$.
But since $u^+v\in E$ and $(u^+v)\alpha=u^+w\notin E$, this is not possible.
Hence, if $N^+(u)$ is finite, it remains to consider the case $n=1=k$.
Due to (\ref{eq_1}), no two vertices of~$X$ have a common predecessor or a common successor.
Thus, also every vertex in~$Y$ has precisely one predecessor and one successor.
Let $v\in Y$ and $w\in X$ with $uv,vw\in E$.
Then we can map the pair $(u,w)$ onto any pair of distinct vertices of~$X$, as $D$ is homogeneous.
Thus, for all $x\neq z\in X$, there exists $y\in Y$ with $xy,yz\in E$.
This shows $|X|=2$ as every vertex of~$D$ has precisely one successor.
Hence, $D$ is a directed cycle of length~$4$, which is isomorphic to $M_2$.

Analogous argumentations in the cases of finite $N^-(u)$, $N^+(v)$ or $N^-(v)$ with $u\in X$ and $v\in Y$ show that the only remaining case is that every vertex in~$D$ has infinite in- and infinite out-neighbourhood.
Due to Lemma~\ref{lem_hom2partite2}, we know that $|u^\perp|$ is either $0$ or infinite and that $|v^\perp|$ is either~$0$ or infinite.
Since $x^\perp\neq\es$ if and only if $y^\perp\neq\es$ for all $x\in X$ and $y\in Y$, the assertion follows from Lemmas~\ref{lem_hom2partite1} and~\ref{lem_hom2partite3}.}
\end{proof}

\bibliographystyle{amsplain}
\bibliography{Bibs}

\providecommand{\bysame}{\leavevmode\hbox to3em{\hrulefill}\thinspace}
\providecommand{\MR}{\relax\ifhmode\unskip\space\fi MR }
\providecommand{\MRhref}[2]{%
  \href{http://www.ams.org/mathscinet-getitem?mr=#1}{#2}
}
\providecommand{\href}[2]{#2}
\begin{thebibliography}{10}

\bibitem{Cherlin-OrderedGraphs}
G.~Cherlin, \emph{The classification of homogeneous ordered graphs}, in
  preparation, 2013.

\bibitem{Cherlin-HomImprimitive}
G.L. Cherlin, \emph{Homogeneous directed graphs. the imprimitive case}, Logic
  colloquium '85 ({O}rsay, 1985), Stud.\ Logic Found.\ Math., vol. 122,
  North-Holland, Amsterdam, 1987, pp.~67--88.

\bibitem{Cherlin-CountHomDigraphs}
\bysame, \emph{The classification of countable homogeneous directed graphs and
  countable homogeneous $n$-tournaments}, vol. 131, Mem.\ Amer.\ Math.\ Soc.,
  no. 621, Amer.\ Math.\ Soc., 1998.

\bibitem{Fraisse1953}
R.~Fra\"iss\'e, \emph{Sur certain relations qui g{\'e}n{\'e}ralisent l'ordre
  des nombre rationnels}, C.\ R.\ Acad.\ Sci.\ Paris \textbf{237} (1953),
  540--542.

\bibitem{F-TheoryOfRelations}
\bysame, \emph{Theory of {R}elations}, Revised edition. With an appendix by
  Norbert Sauer, Stud.\ Logic Found.\ Math., vol. 145, North-Holland Publishing
  Co., Amsterdam, 2000.

\bibitem{Gard-HomogeneousGraphs}
A.~Gardiner, \emph{Homogeneous graphs}, J.\ Combin.\ Theory (Series B)
  \textbf{20} (1976), no.~1, 94--102.

\bibitem{GGK}
M.~Goldstern, R.~Grossberg, and M.~Kojman, \emph{Infinite homogeneous bipartite
  graphs with unequal sides}, Discrete Math. \textbf{149} (1996), no.~1-3,
  69--82.

\bibitem{CountConHomDi}
M.~Hamann, \emph{Countable connected-homogeneous digraphs}, in preparation.

\bibitem{L-FiniteHomDigraphs}
A.H. Lachlan, \emph{Finite homogeneous simple digraphs}, Proceedings of the
  {H}erbrand symposium ({M}arseilles, 1981) (J.~Stern, ed.), Stud.\ Logic
  Found.\ Math., vol. 107, North-Holland, 1982, pp.~189--208.

\bibitem{L-Tournaments}
\bysame, \emph{Countable homogeneous tournaments}, Trans.\ Am.\ Math.\ Soc.
  \textbf{284} (1984), no.~2, 431--461.

\bibitem{LW-CountUltrahomGraphs}
A.H. Lachlan and R.~Woodrow, \emph{Countable ultrahomogeneous undirected
  graphs}, Trans.\ Am.\ Math.\ Soc. \textbf{262} (1980), no.~1, 51--94.

\bibitem{Macpherson-Survey}
H.D. Macpherson, \emph{A survey of homogeneous structures}, Discrete Math.
  \textbf{311} (2011), no.~15, 1599--1634.

\bibitem{Schmerl-HomogeneousPO}
J.H. Schmerl, \emph{Countable homogeneous partially ordered sets}, Algebra
  Universalis \textbf{9} (1979), no.~3, 317--321.

\end{thebibliography}

\end{document}